\documentclass[11pt]{amsart}
\usepackage[T1]{fontenc}
\usepackage[english]{babel}
\usepackage{amscd,amsmath,amsthm,amssymb,graphics}
\usepackage{lmodern,pst-node}
\usepackage{pstcol,pst-plot,pst-3d}
\usepackage{multicol}
\usepackage{epic,eepic}
\usepackage{amsfonts,amssymb,amscd,amsmath,enumitem,verbatim}
\psset{unit=0.7cm,linewidth=0.8pt,arrowsize=2.5pt 4}

\newpsstyle{fatline}{linewidth=1.5pt}
\newpsstyle{fyp}{fillstyle=solid,fillcolor=verylight}
\definecolor{verylight}{gray}{0.97}
\definecolor{light}{gray}{0.9}
\definecolor{medium}{gray}{0.85}
\definecolor{dark}{gray}{0.6}

\def\frk{\frak}               

\def\Phi{{\frk n}}
\def\Phi{{\frk N}}
\def\opn#1#2{\def#1{\operatorname{#2}}} 
\opn\chara{char} \opn\length{\ell} \opn\pd{pd} \opn\rk{rk}
\opn\projdim{proj\,dim} \opn\injdim{inj\,dim} \opn\rank{rank}
\opn\depth{depth} \opn\grade{grade} \opn\height{height}
\opn\embdim{emb\,dim} \opn\codim{codim}

\opn\Tr{Tr} \opn\bigrank{big\,rank}
\opn\superheight{superheight}\opn\lcm{lcm}
\opn\trdeg{tr\,deg}
\opn\reg{reg} \opn\lreg{lreg} \opn\ini{in} \opn\lpd{lpd}
\opn\size{size}\opn\bigsize{bigsize}
\opn\cosize{cosize}\opn\bigcosize{bigcosize}
\opn\sdepth{sdepth}\opn\sreg{sreg}
\opn\link{link}\opn\fdepth{fdepth}
\opn\deg{deg}
\opn\max{max}
\opn\indeg{indeg}
\opn\min{min}
\opn\psln{psln}
\opn\div{div} \opn\Div{Div} \opn\cl{cl} \opn\Cl{Cl}
\let\epsilon\varepsilon
\let\phi=\varphi
\let\kappa=\varkappa

\opn\Spec{Spec} \opn\Supp{Supp} \opn\supp{supp} \opn\Sing{Sing}
\opn\Ass{Ass} \opn\Min{Min}\opn\Mon{Mon} \opn\dstab{dstab} \opn\astab{astab}
\opn\Syz{Syz}
\opn\Ann{Ann} \opn\Rad{Rad} \opn\Soc{Soc}
\opn\Im{Im} \opn\Ker{Ker} \opn\Coker{Coker} \opn\Am{Am}
\opn\Hom{Hom} \opn\Tor{Tor} \opn\Ext{Ext} \opn\End{End}
\opn\Aut{Aut} \opn\id{id}

\opn\nat{nat}
\opn\pff{pf}
\opn\Pf{Pf} \opn\GL{GL} \opn\SL{SL} \opn\mod{mod} \opn\ord{ord}
\opn\Gin{Gin} \opn\Hilb{Hilb}\opn\sort{sort}
\opn\initial{init}
\opn\ende{end}
\opn\height{height}
\opn\depth{depth}
\opn\type{type}
\opn\ldim{ldim}

\opn\aff{aff} \opn\con{conv} \opn\relint{relint} \opn\st{st}
\opn\lk{lk} \opn\cn{cn} \opn\core{core} \opn\vol{vol}
\opn\link{link} \opn\star{star}\opn\lex{lex}
\opn\gr{gr}

\def\pot#1#2{#1[\kern-0.28ex[#2]\kern-0.28ex]}

\opn\dirlim{\underrightarrow{\lim}}
\opn\inivlim{\underleftarrow{\lim}}
\let\to=\rightarrow

\def\Implies{\ifmmode\Longrightarrow \else
        \unskip${}\Longrightarrow{}$\ignorespaces\fi}
\def\implies{\ifmmode\Rightarrow \else
        \unskip${}\Rightarrow{}$\ignorespaces\fi}
\def\iff{\ifmmode\Longleftrightarrow \else
        \unskip${}\Longleftrightarrow{}$\ignorespaces\fi}

\let\:=\colon
 \theoremstyle{plain}
\newtheorem{Theorem}{Theorem}[section]
 \newtheorem{Lemma}[Theorem]{Lemma}
 
 \newtheorem{Proposition}[Theorem]{Proposition}

 \theoremstyle{definition}

 \newtheorem{Example}[Theorem]{Example}

\let\epsilon\varepsilon
\let\kappa=\varkappa
\def\qed{\ifhmode\textqed\fi
      \ifmmode\ifinner\quad\qedsymbol\else\dispqed\fi\fi}
\def\textqed{\unskip\nobreak\penalty50
       \hskip2em\hbox{}\nobreak\hfil\qedsymbol
       \parfillskip=0pt \finalhyphendemerits=0}
\def\dispqed{\rlap{\qquad\qedsymbol}}

\opn\dis{dis}
\def\pnt{{\raise0.5mm\hbox{\large\bf.}}}

\opn\Lex{Lex}

\begin{document}

\author[Mafi and Naderi]{ Amir Mafi and Dler Naderi}
\title{Results on the Hilbert coefficients and reduction numbers}

\address{Amir Mafi, Department of Mathematics, University Of Kurdistan, P.O. Box: 416, Sanandaj, Iran.}
\email{A\_Mafi@ipm.ir}
\address{Dler Naderi, Department of Mathematics, University of Kurdistan, P.O. Box: 416, Sanandaj,
Iran.}
\email{dler.naderi65@gmail.com}

\begin{abstract}
Let $(R,\frak{m})$ be a $d$-dimensional Cohen-Macaulay local ring, $I$ an $\frak{m}$-primary ideal and $J$ a minimal reduction of $I$.
In this paper we study the independence of reduction ideals and the behavior of the higher Hilbert coefficients. In addition, we give some examples in this regards.

\end{abstract}

\subjclass[2010]{13A30, 13D40, 13H10}
\keywords{Hilbert coefficient, Minimal reduction, Associated graded ring}

\maketitle

\section*{Introduction}
Throughout this paper, we assume that $(R,\frak{m})$ is a Cohen-Macaulay local ring of dimension $d$ and the residue class field $R/{\frak{m}}$ is infinite.
For an $R$-module $M$, let $\lambda(M)$ denote the length of $M$. Let $I$ be an $\frak{m}$-primary ideal of $R$. The \emph{Hilbert-Samuel function} $H_I(n)$ of $I$ is defined as $H_I(n)=\lambda(R/I^n)$. There exists a polynomial $P_I(x)$ of the form $$P_I(x)=e_0{{x+d-1}\choose{d}}-e_1{{x+d-2}\choose{d-1}}+...+(-1)^de_d$$
such that $P_I(n)=H_I(n)$ for all large $n$, where $e_i=e_i(I)\in\mathbb{Z}$ are called the \emph{Hilbert coefficients} of $R$ with respect to $I$.

An ideal $J\subseteq I$ is called a reduction ideal of $I$ if $I^{r+1}=JI^r$ for some nonnegative integer $r$ (see \cite{NR}). The least such $r$ is called the reduction number of $I$ with respect to $J$ and denoted by $r_J(I)$. A reduction ideal $J$ is called a minimal reduction if it does not properly contain a reduction ideal of $I$, under our assumption it is generated by a regular sequence. The reduction number of $I$ is defined as

$r(I)=\min\{r_J(I): J$ is a minimal reduction ideal of $I\}$. The reduction number $r(I)$ is said to be independent if $r(I)=r_J(I)$ for all minimal reductions $J$ of $I$.
Sally in \cite{S1} raised the following question: If $(R,\frak{m})$ is a $d$-dimensional Cohen-Macaulay local ring having an infinite residue field, then is $r(\frak{m})$
independent? A natural extension of this question is to replace $r(\frak{m})$ with $r(I)$. Let $G(I)=\bigoplus_{n\geq 0}I^n/I^{n+1}$ be the associated graded ring of $I$. Huckaba in \cite{H} and Trung in \cite{T} independently proved that if $\depth G(I)\geq{d-1}$, then $r(I)$ is independent (see also \cite{M}, \cite{HO}, \cite{HO1}, \cite{ST}, \cite{MN} and \cite{MNO}). Moreover, Wu in \cite{Wu} proved that if $\depth G(I)\geq{d-2}$ and $r(I)\geq n(I) +d-1$; where $n(I)$ is the postulation number of $I$, then $r(I)$ is independent.
However if $d\geq 2$ and $\depth G(I)\leq{d-2}$, then $r(I)$ is not independent in general. Counter-examples have been obtained in \cite{H}, \cite{M} and \cite{MNO}.

It is known that $e_0, e_1$ and $e_2$ are nonnegative integers. Unfortunately, the good behavior of the Hilbert coefficients stops with $e_2$.
Indeed, Narita in \cite{N} showed that it is possible for $e_3$ to be negative (see also \cite{CPR}). However, Itoh in \cite{I1} proved that if $I$ is normal ideal; that is $\overline{I^n}=I^n$ for all positive integer $n$, then $e_3$ is a nonnegative integer (see also \cite{CPR}). Puthenpurakal in \cite{TON} obtained remarkable results about negativity of $e_3$.

The main purpose of this paper is to study the independence of reduction number and also the behavior of the higher Hilbert coefficients. In the last section we collect some examples which disprove a question one can make about the behavior of Hilbert coefficients.

\section{Main results }

We begin this section by recalling some known definitions, notations and results in \cite{HM} and \cite{H1}.
An element $x \in I\setminus I^2$ is said to be superficial for $I$ if there is an integer $c$ such that $(I^{n+1} :x)\cap I^{c} = I^{n}$ for all $n \geq c$. If $\grade (I)\geq 1$ and $x$ is a superficial element, then $x$ is a regular element of $R$ and so by Artin-Rees Theorem  $I^{n+1} :x = I^{n}$ for all $n \gg 0$. If $R/{\frak{m}}$ is infinite, then a superficial element for $I$ always exists.
A sequence $x_1, ..., x_s$ is called a superficial sequence for $I$ if $x_1$ is superficial for $I$ and $x_i$ is superficial for $I/(x_1, ..., x_{i-1})$ for $2 \leq i\leq s$. If $I$ is an $\frak{m}$-primary ideal and $J$ is a minimal reduction of $I$, then there is a superficial sequence $x_1, ..., x_d$ in $I$ such that $J=(x_1, ..., x_d)$. For any element $x\in I$ we let $x^{*}$ denote the image of $x$ in $I/I^2$. We note that if $x^{*}$ is a regular element of $G(I)$, then $x$ is a regular element of $R$ and $G(I/(x))= G(I)/(x^{*})$.

Huckaba and Marley in \cite{HM} constructed the complex $C_.(x_1, ... ,x_d, n)$ which has the following form
\[0 \to {R \mathord{\left/
 {\vphantom {R {{I^{n - d}}}}} \right.
 \kern-\nulldelimiterspace} {{I^{n - d}}}} \to {({R \mathord{\left/
 {\vphantom {R {{I^{n - d + 1}}}}} \right.
 \kern-\nulldelimiterspace} {{I^{n - d + 1}}}})^d} \to {({R \mathord{\left/
 {\vphantom {R {{I^{n - d + 2}}}}} \right.
 \kern-\nulldelimiterspace} {{I^{n - d + 2}}}})^{\left( {\begin{array}{*{20}{c}}
d\\
2
\end{array}} \right)}} \to \quad\cdots\quad  \to {R \mathord{\left/
 {\vphantom {R {{I^n}}}} \right.
 \kern-\nulldelimiterspace} {{I^n}}} \to 0.\]
 Let $C_.(n) = C_.(x_1, x_2, ..., x_d, n)$ and $C_{.}^{'}(n)= C_.(x_1, x_2, ..., x_{d-1}, n)$. For any $n$ there is an exact sequence of complexes
$$
0 \longrightarrow {C_.}^\prime (n) \longrightarrow {C_.}(n) \longrightarrow {C_.}^\prime (n - 1)[ - 1] \longrightarrow 0,
$$
where ${C_.}^\prime (n - 1)[ - 1]$ is the complex ${C_.}^\prime (n - 1)$ shifted to the left by one degree.
Thus, we have the corresponding long exact sequence of homology modules:
$$
\cdots\longrightarrow {H_i}({C_.}^\prime (n)) \longrightarrow {H_i}({C_.}(n)) \longrightarrow {H_{i - 1}}({C_.}^\prime (n - 1)){\kern 1pt} \; \stackrel{x_d}{\longrightarrow} {H_{i - 1}}({C_.}^\prime (n)) \longrightarrow\cdots.
$$
For  $i \ge 1$, we define
$$
{h_i}: = \sum\limits_{n = 1}^\infty  {\lambda ({H_i}({C_.}(n)))}
$$
and
 $$
 {k_i}: = \sum\limits_{n = 2}^\infty  {(n - 1)\lambda ({H_i}({C_.}(n)))}.
 $$
By \cite[\S 4]{HM}, we have\\

$$
{\Delta^d}[{P_I}(n)-{H_I}(n)]=\lambda(I^n/{I^n\cap J})-\sum\limits_{i = 1}^d {{{( - 1)}^i}\lambda ({H_i}({C_.}(n))}$$ $$=\lambda({{{I^n}}/{{J{I^{n - 1}}}}}) - \sum\limits_{i = 2}^d {{{( - 1)}^i}\lambda ({H_i}({C_.}(n))}
$$
and
$$
{e_i}(I) = \sum\limits_{n = i}^\infty  {\left( {{}_{i - 1}^{n - 1}} \right)} \;{\Delta ^d}[{P_I}(n) - {H_I}(n)].
$$
Hence by combining two previous formulas, we have
\[{e_1}(I)=\sum\limits_{n = 1}^\infty {\lambda ({{{I^n}}}/{{J{I^{n -1}}}})}-\sum\limits_{i = 2}^d {{{( - 1)}^i} {h_i}}=\sum\limits_{n = 1}^\infty{\lambda(I^n/{I^n\cap J})}+\sum\limits_{i = 1}^d {{{( - 1)}^{i-1} {h_i}}} \]
and
\[{e_2}(I)=\sum\limits_{n = 2}^\infty  {(n - 1)\lambda ({{{I^n}}}/{{J{I^{n - 1}}}})}  - \sum\limits_{i = 2}^d {{{( - 1)}^i}{k_i}}. \]
\\

For an ideal $I$ of $R$, let $\overline{I}$ denote the integral closure of $I$ in $R$. That is, $\overline{I}$ is the set of all elements $x$ in $R$ satisfying the equation of the form $x^k+a_1x^{k-1}+...+a_k=0$, where $a_i\in I^i$ for $i=1,2,...,k$. The ideal $I$ is integrally closed when $\overline{I}=I$. Also, the ideal $I$ is said to be asymptotically normal if there exists an integer $n_0\geq 1$ such that $I^n$ is integrally closed for all $n\geq n_0$, for interesting family of asymptotically normal ideals, see \cite[Remark 4.3]{CPR}.

\begin{Proposition} \label{l1}
Let $I$ be an $\frak{m}$-primary integrally closed ideal and $J$ be a minimal reduction of $I$. If $e_2=\lambda ({I^2}/{JI}) +1$, then $G(I)$ is Cohen-Macaulay, $e_2={e_1} - {e_0} + \lambda ({R}/{I})$, $r_J(I) \leq 3$ and $r(I)$ is independent for any minimal reduction $J$ of $I$.
\end{Proposition}

\begin{proof}
By \cite[Theorem 12]{I} we have ${e_1}-{e_0}+\lambda({R}/{I})\leq e_2=\lambda ({I^2}/{JI}) +1$. Since $I$ is integrally closed, we have $I^2\cap J=IJ$. Suppose ${e_1} - {e_0} + \lambda ({R}/{I})\leq\lambda ({I^2}/{JI})$. Thus $\lambda ({I^2}/{JI})+\sum\limits_{n = 3}^d {\lambda ({{{I^n}}}/{{{I^n} \cap J}})} +\sum\limits_{i = 1}^d {{{( - 1)}^{i - 1}}{h_i}} \leq\lambda ({I^2}/{JI})$ and so by \cite[Theorem 3.7]{HM} we have $\sum\limits_{n = 3}^d {\lambda ({{{I^n}}}/{{{I^n} \cap J}})}=0$ and $\sum\limits_{i = 1}^d {{{( - 1)}^{i - 1}}{h_i}} =0$.  Therefore $G(I)$ is Cohen-Macaulay and by \cite[Lemma 3.2]{MNO} $\sum\limits_{i = 1}^d {{{( - 1)}^{i - 1}}{k_i}}=0 $. Since
$\sum\limits_{n = 3}^d {(n - 1)\lambda ({{{I^n}}}/{{{I^n} \cap J}})}=0 $, we have $e_2=\lambda ({I^2}/{JI}) $ which is contradiction by our hypothesis. Therefore ${e_1} - {e_0} + \lambda ({R}/{I})=\lambda ({I^2}/{JI})+1 $ and by \cite[Theorem 3.4]{OR} $I^4=JI^3$  and so $r(I)$ is independent for any minimal reduction $J$ of $I$.
\end{proof}

Corso, Polini and Rossi in \cite[Remark 3.7]{CPR} observed that if $I$ is integrally closed ideal and $e_2=0,1,2$, then $G(I)$ is Cohen-Macaulay (see also \cite{E}, \cite{ERV}, \cite{EV} and \cite{I}). Also they observed that assumption on the ideal $I$ being integrally closed cannot be weakened, see \cite[Example 3.8]{CPR}.
In the following proposition we prove that if $I$ is an $\frak{m}$-primary integrally closed ideal and $e_2=3$, then $\depth G(I)\geq d-2$ and $r(I)$ is independent.

\begin{Proposition} \label{P1}
Let $I$ be an $\frak{m}$-primary integrally closed ideal and $J$ be a minimal reduction of $I$. If ${e_2}(I)=3$,
then $\depth G(I)\geq d-2$ and  $r(I) $ is independent.
\end{Proposition}

\begin{proof}
 Let ${e_2}(I)=3 $. By \cite[Theorem 12]{I} we have ${e_1} - {e_0} + \lambda ({R}/{I}) \leq {e_2}(I)=3 $. If ${e_1} - {e_0}+\lambda ({R}/{I}) \leq 2 $, then by \cite[Lemma 3.15]{MNO},$\depth G(I) \geq d-1$ and  ${r}(I) $ is independent. Hence we assume ${e_1} - {e_0} + \lambda ({R}/{I})=3 $. In this case by \cite[Corollary 4.7]{OR} we have  $\depth G(I)\geq d-2$ and so by \cite[Proposition 3.16]{MNO} $r(I)$ is independent.
\end{proof}

Ratliff and Rush \cite{RR} introduced the ideal $\tilde{I}$, which turns out to be the largest ideal containing $I$ with the same Hilbert coefficients as $I$. In particular one has the inclusions $I\subseteq\tilde{I}\subseteq\overline{I}$, where equalities hold if $I$ is integrally closed.

\begin{Lemma} \label{l2}
Let $(R,\frak{m})$ be a Cohen-Macaulay local ring of dimension $2$. Let $I$ be an $\frak{m}$-primary ideal and $J$ be a minimal reduction of $I$. If $r_J(I) \leq 2$ and $\tilde{I}=I$, then $\depth G(I) \geq 1$.
\end{Lemma}

\begin{proof}
If $r_J(I)\leq 2$, then we have $I^{n} \cap J=JI^{n-1}$ for all $n\geq 3$. Since $\tilde{I}=I$, we have $I^{2} :x =I$ for any superficial element $x$ in $I$. Hence by \cite[Proposition 2.1]{MA}, we have  $\widetilde{I^{n}}=I^{n}$ for any $n\geq 1$ and so $\depth G(I)\geq 1$.
\end{proof}
An ideal $I$ is said to be asymptotically normal if there exists an integer $k\geq 1$ such that $I^n$ is integrally closed for all $n\geq k$. In \cite[Remark 4.3]{CPR} there are interesting examples of asymptotically normal ideals that are not normal.
The following result improves \cite[Proposition 16]{I}.

One of the main results is the next theorem.
\begin{Theorem} \label{P2}
Let $(R,\frak{m})$ be a Cohen-Macaulay local ring of dimension $3$. Let $I$ be an $\frak{m}$-primary ideal and $J$ be a minimal reduction of $I$. Assume that $I$ is an asymptotically normal ideal and $\tilde{I}=I$. Then
$P_I(n)=H_I(n)$ for $n=1,2$ if and only if $r_J(I) \leq 2$.
\end{Theorem}

\begin{proof}
 Let $x$ be superficial element of $I$ and set $A=R/(x)$, $B=IA$ and $C=JA$. Then  $\dim A=2$. Since $I$ is an asymptotically normal ideal by \cite[Corollary 7.11]{TO} $\tilde{B}=B$.
 If $r_J(I)\leq 2$, then $r_C(B)\leq 2$. Therefore by Lemma 1.3 $\depth G(B)\geq 1$ and so by \cite[Lemma 2.2]{HM} $\depth G(I)\geq 2$. Hence by \cite[Theorem 2]{M1} we have $r_J(I)=n(I)+3$. Thus $n(I)\leq -1$ and $P_I(n)=H_I(n)$ for all $n\geq 0$. Conversely,
 If $P_I(n)=H_I(n)$ for $n=1,2$, then by \cite[Lemma 2.3]{Wu} we have $P_B(n)=H_B(n)$ for $n=1,2$. So by \cite[Proposition 16]{I} $r_C(B)\leq 2$  and by Lemma 1.3 $\depth G(B)\geq 1$. Therefore $\depth G(I)\geq 2$  and $r_J(I) \leq 2$.
\end{proof}

Let $\mathcal{R}(I)$ be the Rees-algebra of $I$ and $E$ be an $\mathcal{R}(I)$-module. Then in the following theorem we set $H^i(E)$ to be the $ith$- local cohomology module of $E$ with respect to the maximal homogeneous ideal $\mathcal M=(m,It)$ of $\mathcal{R}(I)$ as the support.

The second main result with application is the following theorem.
\begin{Theorem} \label{P3}
Let $(R,\frak{m})$ be a Cohen-Macaulay local ring of dimension $4$. Let $I$ be an $\frak{m}$-primary ideal and $J$ be a minimal reduction of $I$. If $I$ is an asymptotically normal ideal and $ r_{J}(I)\leq 3 $, then $ e_{4}(I)\leq 0. $
\end{Theorem}

\begin{proof}
By \cite[Lemma 2.4]{HO} for $n\gg 0$ we have
\begin{align*}
{H^i}({G(I^{n}))_j} = 0\;\;  for~ j \geq 1\;and\;i=0, 1, 2, 3, 4
\end{align*}
and since $ r_{J}(I)\leq 3 $ we get by \cite[Proposition 3.2]{T} and \cite[Lemma 2.4]{HO}
\begin{align*}
{H^4}({G({I^n}))_0} = 0.
\end{align*}
Let $q$ be the integer such that $I^{q} $ is normal. By \cite[Theorem 7.3]{TO}, $\depth G(I^{n})\geq 2$ for $n\gg 0$ and so ${H^i}({G(I^{n})})=0$ for $n=0,1$. Thus
\begin{align*}
{a_2}({I^n})<{a_3}({I^n})=0
\end{align*}
and therefore
${a_2}({I^n})\leq -1$
  and
 ${H^2}({G({I^n}))_0} = 0.$

Set $h_{i}=\lambda({H^i}({G({I^n}))_0})$ for $i=0,1,2,3,4.$
Then
\[{h_0} = {h_1} = {h_2} = {h_4} = 0.\]
Set $K:=I^{n}$  for $n\gg 0$ and let
\[P_K(z) = {c_0}\left( {\begin{array}{*{20}{c}}
{z + 3}\\
3
\end{array}} \right) - {c_1}\left( {\begin{array}{*{20}{c}}
{z + 2}\\
2
\end{array}} \right) + {c_2}\left( {\begin{array}{*{20}{c}}
{z + 1}\\
1
\end{array}} \right) - {c_3}\]
be the Hilbert polynomial of $ G(K) $ i.e.
\[P_K(i) = \lambda \left( {{{{K^i}} \mathord{\left/
 {\vphantom {{{K^i}} {{K^{i + 1}}}}} \right.
 \kern-\nulldelimiterspace} {{K^{i + 1}}}}} \right){\kern 1pt} {\kern 1pt} for{\kern 1pt} {\kern 1pt} i \gg 0.\]
 By Grothendieck-Serre formula  we get
 \[{H_K}(i) - {P_K}(i) = \sum\limits_{s = 0}^4 {{{( - 1)}^s}\lambda ({H^s}{{({G}(K))}_i})}. \]
 Set  $ i=0 $ we get
 \[\lambda \left( {{R \mathord{\left/
 {\vphantom {R {{I^r}}}} \right.
 \kern-\nulldelimiterspace} {{I^n}}}} \right) - \left[ {{c_0} - {c_1} + {c_2} - {c_3}} \right] = {h_0} - {h_1} + {h_2} - {h_3} + {h_4} =  - {h_3}.\]
 Let $\phi_I(n)$ be the  Hilbert-Samuel polynomial of  I i.e.,
 \[\phi_I(z) = {e_0}(I)\left( {\begin{array}{*{20}{c}}
{z + 3}\\
4
\end{array}} \right) - {e_1}(I)\left( {\begin{array}{*{20}{c}}
{z + 2}\\
3
\end{array}} \right) + {e_2}(I)\left( {\begin{array}{*{20}{c}}
{z + 1}\\
2
\end{array}} \right) - {e_3}(I)\left( {\begin{array}{*{20}{c}}
{z }\\
1
\end{array}} \right) + {e_4}(I).\]

 Write
 \[\phi_{K} (z) = {c_0}\left( {\begin{array}{*{20}{c}}
{z + 3}\\
4
\end{array}} \right) - {c_1}\left( {\begin{array}{*{20}{c}}
{z + 2}\\
3
\end{array}} \right) + {c_2}\left( {\begin{array}{*{20}{c}}
{z + 1}\\
2
\end{array}} \right) - {c_3}\left( {\begin{array}{*{20}{c}}
z\\
1
\end{array}} \right) + {c_4}.\]
Clearly $ \phi _K(z)=\phi _I(nz) $.
In particular $ c_{4}=e _{4}(I) $.
Also notice that
\[{\phi _K}(1) = {c_0} - {c_1} + {c_2} - {c_3} + {c_4} = {\phi _I}(n) = \lambda \left( {{R \mathord{\left/
 {\vphantom {R {{I^n}}}} \right.
 \kern-\nulldelimiterspace} {{I^n}}}} \right).\]
 So we get $ {c_4} =  - {h_3} $. Thus ${e_4} =  - {h_3} $. Thus $ {e_4}\leq 0 $.
\end{proof}

For any ideal $I$ the set of ideals ${(I^{n+1}:I^n)}$ forms an ascending chain. Let $\widetilde{I}$ denote the union of these ideals. Ratliff and Rush \cite{RR} showed that $\widetilde{I}$ is the largest ideal containing $I$ which has the same Hilbert polynomial as $I$.
 In the following proposition we use notations: ${\mathcal{B} ^I}(x,R)=\bigoplus_{n=0}^{\infty}\frac{I^{n+1}:x}{I^n}$ and $L^I(R)=\bigoplus_{n=0}^{\infty}R/I^{n+1}$.

 \begin{Proposition} \label{P1}
Let $(R, \frak{m})$ be a Cohen-Macaulay local ring of dimension $3$. Let $I$ be an $\frak{m}$-primary ideal and $J$ be a minimal reduction of $I$. If  ${e_2}(I)\leq 1$,then ${e_3}(I) \leq 0$.
\end{Proposition}

\begin{proof}
 By \cite [Proposition 6.4]{TON} we can assume that ${e_2}(I)=1$. Let $x$ be a superficial element of $I$. Set $A=R/(x)$ , $B=IA$ and $C=JA$, then  $\dim A=2$, ${e_i}(I)={e_i}(B)$ for $i=0,1,2$ and ${e_2}(B)=1$. By \cite[Corollary 4.13]{HM} we have
 \[{e_2}(B) = \widetilde {{e_2}(B)}=\sum\limits_{n \ge 1}^{} {n \lambda ({{{\widetilde {B^{n + 1}}}} \mathord{\left/
 {\vphantom {{{B^{n + 1}}} {C  {{B^n}}}}} \right.
 \kern-\nulldelimiterspace} {C \widetilde {{B^n}}}})}=1.\]
 Therefore we have
 \[\sum\limits_{n \ge 2}^{} {\lambda ({{{\widetilde {B^{n + 1}}}} \mathord{\left/
 {\vphantom {{{B^{n + 1}}} {C  {{B^n}}}}} \right.
 \kern-\nulldelimiterspace} {C \widetilde {{B^n}}}})}=0 .\]
 Thus by \cite[Proposition 1.9]{GR} we have
 \[\widetilde {{e_3}(B)} = \sum\limits_{n \ge 2}^{} {\left( {\begin{array}{*{20}{c}}
n\\
2
\end{array}} \right) \lambda ({{{\widetilde {B^{n + 1}}}} \mathord{\left/
 {\vphantom {{{B^{n + 1}}} {C  {{B^n}}}}} \right.
 \kern-\nulldelimiterspace} {C \widetilde {{B^n}}}})}=0. \]

  By using \cite[6.3]{TO} we have the following exact sequence
 \[0 \to {\mathcal{B} ^I}(x,R) \to {H^0}({L^I}(R))(-1) \to {H^0}({L^I}(R)) \to {H^0}({L^I}(A)).\]
 Therefore we obtain
 \[b:= \lambda(\mathcal{B} ^I(x,R)) \leq \lambda({H^0}({L^I}(A)))=:r.\]
 By using \cite[1.5]{TON} we get that ${{e_3}(B)}=\widetilde {{e_3}(B)}-r$ and by \cite[Proposition 1.2]{RV} ${{e_3}}(I)= {{e_3}}(B)+b$
then $ {e_3}(I) \leq 0$.
\end{proof}

As an application of Proposition 1.6 we give the following example.
The computation of examples is performed by using Macaulay2 \cite{GS} and CoCoA \cite{AB}.
For the convenience of the reader we calculate Hilbert series and Hilbert polynomial of the following example.

\begin{Example}\label{e}
Let $R=k[x,y,z]_{(x ,y, z)}$, where $k$ is a field, and let $I=(x^3, y^3, x^2y+z^3, xz^2,y^2z+x^2z)$. Then $\depth G(I)=0$  and we have the following
Hilbert series
\[{P_I}(t) = \frac{{16 + 5t + 5{t^2} -5{t^3} +6{t^4} +10{t^5} -13 {t^6}+2{t^7} + t^8}}{{{{(1 - t)}^3}}},\]
and
Hilbert polynomial
\[{P_I}(n) = 27\left( {\begin{array}{*{20}{c}}
{n + 2}\\
3
\end{array}} \right)
-18\left( {\begin{array}{*{20}{c}}
{n + 1}\\
2
\end{array}} \right)
+\left( {\begin{array}{*{20}{c}}
n \\
1
\end{array}} \right)
  +15.\]

Hence $e_2(I)=1$ and $e_3(I)\leq 0$.

\end{Example}

\begin{proof}
Let
 \[P_I (t) =f(t)/ (1 - t)^3 \]
be the Hilbert series of the ideal $I$ and
\[ f(t)=a_{0} + a_{1}t + a_{2}t^2 +a_{3}t^3 + a_{4}t^4 + a_{5}t^5 + a_{6}t^6 + a_{7}t^7 + a_{8}t^8.\]
Then \[a_{0}=\lambda(R/I) , \quad \quad a_{i}=\lambda(I^{i}/I^{i+1})- \sum\limits_{n = 0}^{i - 1} {\left( {\begin{array}{*{20}{c}}
{d + n}\\
{d - 1}
\end{array}} \right)} {a_{i - 1 - n}},\]
where $i$ is a non-negative integer.\\
Therefore by using Macaulay2, we have $\lambda(R/I)=16, \lambda(I/I^2)=53, \lambda(I^2/I^3)=116, \lambda(I^3/I^4)=200,
 \lambda(I^4/I^5)=311, \lambda(I^5/I^6)=459, \lambda(I^6/I^7)=631, \lambda(I^7/I^8)=829, \lambda(I^8/I^9)=1054, \lambda(I^9/I^{10})=1306, \lambda(I^{10}/I^{11})=1585$ and so on. Hence we can obtain the following
\[ a_0=\lambda(R/I)=16\]
\[ a_{1}=\lambda(I/I^2)-3\lambda(R/I)=5\]
\[a_2=\lambda(I^2/I^3)-6\lambda(I/I^2)-3\lambda(R/I)=5\]
and also by the above formula we have $a_3=-5, a_4=6, a_5=10, a_6=-13, a_7=2, a_8=1, a_9=0, a_{10}=0.$\\
For the computing of the Hilbert polynomial, we have
$e_{0}=f(1)=27, e_{1}=f^{\prime}(1)/1!=18,e_{2}=f^{\prime \prime}(1)/2!=1, e_{3}(I)=f^{(3)}(1)/3!=-15$.
This completes the proof.
\end{proof}

\begin{Theorem} \label{P1}
Let $(R,\frak{m})$ be a Cohen-Macaulay local ring of dimension $3$, $I$ an $\frak{m}$-primary integrally closed ideal and $J$ a minimal reduction of $I$. If ${e_1}(I) - {e_0}(I) + \lambda(R/I)={e_2}(I)$, then ${e_3}(I) \leq 0$.
\end{Theorem}

\begin{proof}
  By \cite[Lemma 11]{I} there exists a superficial element $x$ of $I$ so that $I/(x)$ is an integrally closed ideal of $A=R/(x)$. Set $B=IA$ and $C=JA$. Then  $\dim A=2$, ${e_i}(I)={e_i}(B)$ for $i=0 , 1, 2$ and ${e_1}(B) - {e_0}(B) + \lambda(A/B)={e_2}(B)$. By \cite[Corollary 4.13]{HM} we have
 \[{e_1}(B) - {e_0}(B) + \lambda(A/B)=\widetilde {{e_1}(B)} - \widetilde {{e_0}(B)} + \lambda(A/B)=\sum\limits_{n \ge 1}^{} {\lambda ({{{\widetilde {B^{n + 1}}}} \mathord{\left/
 {\vphantom {{{B^{n + 1}}} {C  {{B^n}}}}} \right.
 \kern-\nulldelimiterspace} {C \widetilde {{B^n}}}})} \]
 and
 \[{e_2}(B) = \widetilde {{e_2}(B)}=\sum\limits_{n \ge 1}^{} {n \lambda ({{{\widetilde {B^{n + 1}}}} \mathord{\left/
 {\vphantom {{{B^{n + 1}}} {C  {{B^n}}}}} \right.
 \kern-\nulldelimiterspace} {C \widetilde {{B^n}}}})}. \]
Therefore
 \[\sum\limits_{n \ge 2}^{} {\lambda ({{{\widetilde {B^{n + 1}}}} \mathord{\left/
 {\vphantom {{{B^{n + 1}}} {C  {{B^n}}}}} \right.
 \kern-\nulldelimiterspace} {C \widetilde {{B^n}}}})}=0 \]
 and so
 \[\widetilde {{e_3}(B)} = \sum\limits_{n \ge 2}^{} {\left( {\begin{array}{*{20}{c}}
n\\
2
\end{array}} \right) \lambda ({{{\widetilde {B^{n + 1}}}} \mathord{\left/
 {\vphantom {{{B^{n + 1}}} {C  {{B^n}}}}} \right.
 \kern-\nulldelimiterspace} {C \widetilde {{B^n}}}})}=0. \]

  By using \cite[6.3]{TO} we have the following exact sequence
 \[0 \to {\mathcal{B} ^I}(x,R) \to {H^0}({L^I}(R))( - 1) \to {H^0}({L^I}(R)) \to {H^0}({L^I}(A)).\]
 Hence
 \[b:= \lambda(\mathcal{B} ^I(x,R)) \leq \lambda({H^0}({L^I}(A)))=:r. \]
 by using \cite[1.5]{TON} we get that ${e_3}(B)=\widetilde {{e_3}(B)} -r$ and so by \cite[Proposition 1.2]{RV} ${e_3}(I)= {e_3}(B)+b$.
Thus ${e_3}(I) \leq 0$.
\end{proof}

As an application of Theorem 1.8 we give the following example.
\begin{Example}\label{e}
Let $R=k[x,y, z,u, v, w]_{(x,y, z,u, v, w)}$, where $k$ is a field. Let $Q=(z^2,zu, zv, uv, u^3-yz, v^3-xz)$, $\frak{m}=(x,y, z,u, v, w)$ and $S=R/Q$. Then $S$ is a 3-dimensional
Cohen-Macaulay local ring and $\frak{m}$ is a maximal ideal (integrally closed  $\frak{m}$-primary ideal) of S. Then $\depth G(I)=1$  and we have the following Hilbert series
\[{P_I}(t) = \frac{{1+ 3t +3{t^3} -{t^4}}}{{{{(1 - t)}^3}}},\]

and Hilbert polynomial
\[{P_I}(n) = 6\left( {\begin{array}{*{20}{c}}
{n + 2}\\
3
\end{array}} \right)
-8\left( {\begin{array}{*{20}{c}}
{n + 1}\\
2
\end{array}} \right)
+3\left( {\begin{array}{*{20}{c}}
n \\
1
\end{array}} \right)
  +1.\]
Thus $e_1-e_0+ \lambda(R/I)=e_2$ and $e_3\leq 0$.
\end{Example}

\section{Examples}

Marley in \cite{M1} proved that if $(R,\frak{m})$ is a $d$-dimensional Cohen-Macaulay local ring and $I$ an $\frak{m}$-primary ideal such that $\depth G(I)\geq d-1$, then all Hilbert coefficients $e_0(I),e_1(I),...,e_d(I)$ are non-negative integers.

Thus it is natural to ask if
 $d-3\leq\depth G(I)=t\leq d-2$ whenever $d\geq 3$, then is $e_{i}(I)\geq 0$ for $i=t,t+1,...,d$? In the following examples we show that the question is negative.
\begin{Example}\label{e}
Let $R = k[x,y,z,u]_{(x,y,z,u)}$, where $k$ is a field, and \\ $I=(x^3, y^3, z^3, u^3, xy^2, yz^2, zu^2, xyz, xyu)$. Then we have $\depth G(I)=2$,
 Hilbert series
\[{P_I}(t) = \frac{{33 + 19t + 21{t^2} + 7{t^3} + 5{t^4} - 3{t^5} - {t^6}}}{{{{(1 - t)}^4}}}\]
and Hilbert polynomial
\[{P_I}(n) = 81\left( {\begin{array}{*{20}{c}}
{n + 3}\\
4
\end{array}} \right)
-81\left( {\begin{array}{*{20}{c}}
{n + 2}\\
3
\end{array}} \right)
+27\left( {\begin{array}{*{20}{c}}
{n + 1}\\
2
\end{array}} \right)
+23\left( {\begin{array}{*{20}{c}}
n \\
1
\end{array}} \right)
  -50.\]

Hence ${e_0}(I)=81$, ${e_1}(I)=81$, ${e_2}(I)=27$, ${e_3}(I)=-23$ and ${e_4}(I)=-50$.

\end{Example}
In the following example we show that if $(R,\frak{m})$ is a $5$-dimensional Cohen-Macaulay local ring and $I$ an $\frak{m}$-primary ideal and $\depth G(I)=3$, then $e_{i}(I) < 0$ for $i=3,4,5$.
\begin{Example}\label{e}
Let $R = k[x,y,z,u,v]_{(x,y,z,u,v)}$, where $k$ is a field and \\ $I=(x^3, y^3, z^3, u^2, v,  xy^2, yz^2, xyz, xyu)$. Then we have  $\depth G(I)=3$,
 Hilbert series
\[{P_I}(t) = \frac{{28 + 11t + 10{t^2} + 5{t^3} + {t^4} - {t^5} }}{{{{(1 - t)}^5}}}\]

 and Hilbert polynomial
\[{P_I}(n) = 54\left( {\begin{array}{*{20}{c}}
{n + 4}\\
5
\end{array}} \right)
-45\left( {\begin{array}{*{20}{c}}
{n + 3}\\
4
\end{array}} \right)
+21\left( {\begin{array}{*{20}{c}}
{n + 2}\\
3
\end{array}} \right)
+\left( {\begin{array}{*{20}{c}}
{n + 1}\\
2
\end{array}} \right)
-4\left( {\begin{array}{*{20}{c}}
n \\
1
\end{array}} \right)
  +1.\]
Therefore ${e_0}(I)=54$, ${e_1}(I)=45$, ${e_2}(I)=21$, ${e_3}(I)=-1$, ${e_4}(I)=-4$ and ${e_5}(I)=-1$.

\end{Example}
In the following example we show that if $(R,\frak{m})$ is a $4$-dimensional Cohen-Macaulay local ring and $I$ an $\frak{m}$-primary ideal, $\depth G(I)=2$ and
$e_{3}(I) > 0$, then  $e_{4}(I) < 0$.
\begin{Example}\label{e}
Let $R = k[x,y,z,u]_{(x,y,z,u)}$, where $k$ is a field and\\ $I=(x^4, y^4, z^4, u^4, x^2y^2, y^2z^2, z^2 u^2, xyz, xyu)$. Then we have  $\depth G(I)=2$,
Hilbert series
\[{P_I}(t) = \frac{{81 + 58t + 31{t^2} + 7{t^3} -{t^4} }}{{{{(1 - t)}^4}}}\]
and Hilbert polynomial
\[{P_I}(n) = 176\left( {\begin{array}{*{20}{c}}
{n + 3}\\
4
\end{array}} \right)
-137\left( {\begin{array}{*{20}{c}}
{n + 2}\\
3
\end{array}} \right)
+46\left( {\begin{array}{*{20}{c}}
{n + 1}\\
2
\end{array}} \right)
-3\left( {\begin{array}{*{20}{c}}
n \\
1
\end{array}} \right)
  -1.\]
So we have ${e_0}(I)=176$, ${e_1}(I)=137$, ${e_2}(I)=46$, ${e_3}(I)=3$ and ${e_4}(I)=-1$.
\end{Example}

In the following example we show that if $(R,\frak{m})$ is a $4$-dimensional Cohen-Macaulay local ring and $I$ an $\frak{m}$-primary ideal, $\depth G(I)=1$
then $e_{4}(I) > 0$ but  $e_{3}(I) < 0$.
\begin{Example}\label{e}
Let $R = k[x,y,z,u]_{(x,y,z,u)}$, where $k$ is a field and \\ $I=(x^3, y^3, z^3, u^3, xy^2, xz^2, x u^2, xyz, xyu)$. Then we have  $\depth G(I)=1$,
 Hilbert series
\[{P_I}(t) = \frac{{37 + 14t + 17{t^2} + 15{t^3} + 6{t^4} - 12{t^5} +4{t^6}}}{{{{(1 - t)}^4}}}\]
 and Hilbert polynomial
\[{P_I}(n) = 81 \left( {\begin{array}{*{20}{c}}
{n + 3}\\
4
\end{array}} \right)
-81\left( {\begin{array}{*{20}{c}}
{n + 2}\\
3
\end{array}} \right)
+38\left( {\begin{array}{*{20}{c}}
{n + 1}\\
2
\end{array}} \right)
+\left( {\begin{array}{*{20}{c}}
n \\
1
\end{array}} \right)
  +6.\]

Thus ${e_0}(I)=81$, ${e_1}(I)=81$, ${e_2}(I)=38$, ${e_3}(I)=-1$ and ${e_4}(I)=6$.

\end{Example}

In the following example $(R,\frak{m})$ is a $5$-dimensional Cohen-Macaulay local ring, $I$ an $\frak{m}$-primary ideal, $\depth G(I)=3$ and $e_{4}(I) < 0$ but $e_{5}(I) \geq 0$.
\begin{Example}\label{e}
Let $R = k[x,y,z,u,v]_{(x,y,z,u,v)}$, where $k$ is a field and\\ $I=(x^4, y^4, z^4, u^4, v,  x^2y^2, y^2z^2,z^2u^2, xyz, xyu)$. Then $\depth G(I)=3$,
Hilbert series
\[{P_I}(t) = \frac{{81 + 58 t + 31 {t^2} + 7 {t^3} - {t^4}  }}{{{{(1 - t)}^5}}}\]
 and Hilbert polynomial
\[{P_I}(n) = 176 \left( {\begin{array}{*{20}{c}}
{n + 4}\\
5
\end{array}} \right)
-137 \left( {\begin{array}{*{20}{c}}
{n + 3}\\
4
\end{array}} \right)
+46\left( {\begin{array}{*{20}{c}}
{n + 2}\\
3
\end{array}} \right)
-3 \left( {\begin{array}{*{20}{c}}
{n + 1}\\
2
\end{array}} \right)
+ \left( {\begin{array}{*{20}{c}}
n \\
1
\end{array}} \right)
  .\]

Thus ${e_0}(I)=176$, ${e_1}(I)=137$, ${e_2}(I)=46$, ${e_3}(I)=3$, ${e_4}(I)=-1$ and ${e_5}(I)=0$.

\end{Example}

In the following example we prove that if $(R,\frak{m})$ is a $d$-dimensional Cohen-Macaulay local ring and $I$ an $\frak{m}$-primary ideal, then all Hilbert coefficients are positive however $\depth G(I) \leq d-2$.
\begin{Example}\label{e}
Let $R = k[x,y,z]_{(x,y,z)}$, where $k$ is a field, and $I=(x^4, y^4, z^4, x^3y, y^3z, xyz)$. Then we have  $\depth G(I)=1$,
Hilbert series
\[{P_I}(t) = \frac{{31 + 9t + 7{t^2} + {t^3} }}{{{{(1 - t)}^3}}}\]

and Hilbert polynomial
\[{P_I}(n) = 48\left( {\begin{array}{*{20}{c}}
{n + 2}\\
3
\end{array}} \right)
-26\left( {\begin{array}{*{20}{c}}
{n + 1}\\
2
\end{array}} \right)
+10\left( {\begin{array}{*{20}{c}}
n \\
1
\end{array}} \right)
  -1.\]

Thus all Hilbert coefficients are positive.

\end{Example}

In the following example we show that $\depth G(I)=0$ but all Hilbert coefficients are non negative.
\begin{Example}\label{e}
Let $R = k[x,y,z]_{(x,y,z)}$, where $k$ is a field and $I=(x^4, y^4, z^4,x^3y,xy^3,y^3z,yz^3)$. Then $\depth G(I)=0$,
Hilbert series
\[{P_I}(t) = \frac{{30 + 12t + 22{t^2} + 8{t^3} -2{t^4} - 12{t^5} +6 {t^6}}}{{{{(1 - t)}^3}}}\]
 and Hilbert polynomial
\[{P_I}(n) = 64\left( {\begin{array}{*{20}{c}}
{n + 2}\\
3
\end{array}} \right)
-48\left( {\begin{array}{*{20}{c}}
{n + 1}\\
2
\end{array}} \right)
+4\left( {\begin{array}{*{20}{c}}
n \\
1
\end{array}} \right)
  .\]

Hence all Hilbert coefficients are no negative.

\end{Example}

\subsection*{Acknowledgements}
We would like to thank the referee for a careful reading of the manuscript and for providing
helpful suggestions.


\end{document}